\documentclass[11pt,leqno]{article}

\usepackage{amssymb,amsmath,amsthm,mysects,url,rotating}
 \usepackage{graphicx,epsfig}
 \usepackage{xcolor,graphicx}
\newcommand{\vp}{\varepsilon}
\newcommand{\bb}[1]{\mathbb{#1}}
\newcommand{\cl}[1]{\mathcal{#1}}

\newcommand{\ovl}{\overline}

\theoremstyle{plain}
\newtheorem{thm}{Theorem}[section]
\newtheorem{lem}[thm]{Lemma}

\newtheorem*{slem}{Sublemma}
\newtheorem{pro}[thm]{Proposition}

\newtheorem{cor}[thm]{Corollary}

\theoremstyle{definition}

\newtheorem{dfn}[thm]{Definition}

\theoremstyle{remark}
\newtheorem{rem}[thm]{Remark}

\numberwithin{equation}{section}

\def\RR{\bb R}
\def\R{\bb R}
\def\Z{\bb Z}

\def\C{\bb C}

\def\E{\bb E}
\def\EE{\bb E}
\def\N{\bb N}
\def\P{\bb P}

\def\T{\bb T}

\def\hat{\widehat}
\def\nl{\nolimits}
\begin{document}
\def\d{\delta}

 \title{
  Random unitaries,  amenable linear groups and Jordan's theorem}

\author{by\\
Emmanuel Breuillard\footnote{breuilla@uni-muenster.de}\\
WWU M\"unster\\ 
 and\\
Gilles  Pisier\footnote{pisier@math.tamu.edu}\\
Texas A\&M University and UPMC-Paris VI}

\maketitle
\begin{abstract} It is well known that a dense
subgroup $G$ of the complex unitary group $U(d)$ cannot be amenable
as a discrete group when $d>1$. When $d$ is large enough
we give  quantitative
versions of this phenomenon in connection with
certain estimates of random   Fourier series
on the compact group $\bar G$ that is the closure of $G$.
Roughly, we show that if $\bar G$ covers a large enough
part of $U(d)$ in the sense of metric
entropy then $G$ cannot be amenable.
The results are all based on a
version of a classical theorem of Jordan that says 
that if $G$   is finite, or amenable as a discrete group,
then $G$ contains an Abelian subgroup
with index $e^{o(d^2)}$. 
 \end{abstract}
 
 
 \def\tr{{\rm tr}}

\vfill\eject
\noindent  Let $G$ be a compact group.
We denote by $m_G$ the normalized Haar measure on    $G$,
and by $\hat G$ a maximal family of 
mutually distinct (up to unitary equivalence)   irreducible unitary representations on $G$.
For any $\pi\in \hat G$ let
$\chi_\pi(x)=\tr (\pi(x))$ denote as usual its character,
so that $\{\chi_\pi\mid \pi\in \hat G\}$ is an orthonormal system
in $L_2(G)$.\\
Let $M_d$ be the space of  matrices of size $d \times d$ with
complex entries.
We use the standard notation $|a|:=\sqrt{a^*a}$, i.e. the unique non-negative self-adjoint matrix whose square is $a^*a$.\\
Let $\mathbb{T}$ be the unit circle $\mathbb{T}=\{z \in \mathbb{C} ; |z|=1\}$ in the complex plane.
Let $U(d)\subset M_d$ denote  the group of all   unitary matrices.

Our investigation  is motivated by the following   from \cite{Pi2}
(see  \S  \ref{pfs} below):
 
 \begin{thm}[Characterization of Subgaussian characters]\label{t4} Let $(G_n)$ be a sequence of compact groups, 
 let $\pi_n \in \hat{G_n}$ be nontrivial   and let  $\chi_{n}=\chi_{\pi_n}$ as well as
 $d_{n}=d_{\pi_n}$. The following are equivalent:
\item[(i)]   There is a constant $C$ such that
$$\forall n\forall a\in M_{d_n}\quad \tr|a| \le C\sup\nolimits_{g\in G_n} |\tr (a\pi_n(g))|.$$
\item[(ii)] There is $C>0$ such that
$$\forall n\quad \int \exp{ (  |\chi_{ n}/C|^2)} dm_{G_n}\le \exp(1) \ (=e).$$
\item[(iii)] There is a constant $C$ such that
$$\forall n\quad d_n\le C \int_{U(d_n)} \sup\nolimits_{g\in G_n} | \tr ({u}  \pi_n(g)) | m_{U(d_n)}(d{u})  .$$
\end{thm}

The property (i) means that the singletons $\{\pi_n\}\subset \hat G_n$
are Sidon with constant $C$ in the sense defined e.g. in \cite{HR},
while (ii) means that they are central $\Lambda$ Sidon
with a fixed constant in the sense of \cite{Hut}. See \S \ref{sid} for more background on this. Equivalently,
(ii) says that the tail behaviour
of   $\chi_{n}$ is  dominated (uniformly over $n$) by that of a standard Gaussian  normal
random variable. In other words the $\chi_{n}$'s are uniformly subgaussian.
 Using the Taylor expansion of $x\mapsto \exp{x^2}$ and Stirling's formula,
it is easy to check that (ii) is equivalent to:
 There is a constant $C$ such that  
$$\sup\nl_{p\in  2\N}   \| \chi_{ n} \|_p/\sqrt p \le C.$$
 See \S \ref{sp1} for more on this.

What is a bit surprising in the preceding statement
is that the subgaussian integrability property of the character expressed by
(ii) implies a rather strong property of the  \emph{whole range} of $\pi$, that is perhaps
better described as a ``density" property like in the next corollary.

\begin{cor}\label{c1} The preceding properties are equivalent to
\item[(iv)]  There is a number $0\le {\alpha}<\sqrt 2$ such that
for any $n$ and any $u\in U(d_n)$
there is $t\in G_n$ and $z\in \T$ such that
$$\tr | u- z \pi_n(t)|^2 \le {\alpha}^2 d_n .$$
\end{cor}
 \begin{proof} 
        Assume (i). For simplicity let $G=G_n$, $d=d_n$ and $\pi=\pi_n$.
        Then for any $u\in U(d)$ we have
        $1/C \le   \sup\nl_{g\in G, z\in \T} \Re( z d^{-1}\tr(u\pi(g)) )$.
         Equivalently $$\inf\nl_{g\in G, z\in \T} d^{-1}\tr | u-z\pi(g)|^2 \le 2(1-1/C^2),$$
         and hence (iv) holds.\\
         Conversely assume (iv).  
         Then  for any $u\in U(d)$ we have
         $\inf\nl_{g\in G, z\in \T} d^{-1}\tr | u-z\pi(g)|^2 \le {\alpha}^2$, and hence
        $1-{\alpha}^2/2 \le   \sup\nl_{g\in G, z\in \T} \Re( z d^{-1}\tr(u\pi(g)) )=\sup\nl_{g\in G }d^{-1}|\tr(u\pi(g))| $. Thus 
        $$d(1-{\alpha}^2/2) \le \inf\nl_{u\in U(d)} \sup\nl_{g\in G } |\tr(u\pi(g))| .$$
        A fortiori (iii) holds.
        \end{proof}
        \begin{rem}
Note that for any $u,v\in U(d_n)$ there is $z\in \T$ or even $z\in \{-1,1\}$
such that $\tr | u- z v|^2 \le 2  d_n .$ Indeed,
the average over all such $z$'s is equal to $ 2  d_n .$ 
\end{rem}
  \begin{rem}\label{pr50}[On irreducibility] If a unitary representation
  $\pi_n$ satisfies the inequality appearing in  
  Theorem \ref{t4} (i), then it is irreducible.
  Indeed, if $P_1,P_2$ are mutually orthogonal projections
  onto invariant subspaces
  for $\pi_n$, and if $a$ is a matrix such that $P_1 a P_2=a$
  we have $\tr(a\pi_n(t))=0$ for all $t$, and hence
  the inequality in (i) implies $a=0$, so we must have either $P_1=0$ or  $P_2=0$.
\end{rem}
The fundamental example for Corollary \ref{c1} is very simple: just take
$G=\prod U(d_n)$ and let $\pi_n$ be the coordinates on $G$.
In that case, (iv) obviously holds with ${\alpha}=0$.

 Until recently, the second author   believed naively  that 
 the preceding  Theorem \ref{t4} could be applied to finite groups.
  To his surprise, the first author 
  showed him that it is not so 
   (and he showed him Turing's paper
  \cite{Tu} that already 
  invoked Jordan's theorem to emphasize that general phenomenon, back in 1938 !). 
  The reason is  roughly that any ``large" \emph{finite} subgroup $G\subset U(d)$
  contains a ``large" Abelian subgroup $\Gamma\subset G$ (and even a normal one),
  with an upper bound for the index,
  namely $[G: \Gamma]\le \exp{o(d^2)}$ that contradicts the density
  expressed in (iv), except for the trivial case when $d_n$ stays bounded.
  More precisely, 
 the root for this lies
  in  a Theorem of Camille Jordan  from 1878:\\
  \begin{thm}\label{b1}    Any finite subgroup of $U(d)$ has a normal Abelian subgroup of index
  bounded by a function $f(d)$ depending only on $d$.\end{thm}

  We will show in Theorem \ref{t10}
that the bound $f(d)\le (d+1)! =\exp{ O(d\log(d))}$ (see below) implies
   for any representation $\pi :\ G \to U(d)$ with \emph{finite} or \emph{amenable} range
  $$\int_{U(d)} \sup\nolimits_{g\in G} | \tr ({u}  \pi(g)) | m_{U(d)}(d{u})=O( (d\log(d))^{1/2}) , $$
  and 
  $$\int \exp    { (  |\chi_\pi/C|^2)} dm_G \le e \Rightarrow (1/C) =O(  (\log d/d)^{1/2} )$$ 
  and these orders of growth are sharp.
  
  Thus we cannot have a sequence of finite groups $G_n$
  satisfying the properties in Theorem \ref{t4} or Corollary \ref{c1} unless the dimensions $d_n$ stay bounded.
  
  Similar questions have been considered previously in the theory of  Sidon sets in duals of non-commutative compact groups. We describe this connection in \S \ref{sid}. When the representations $\pi_n$ are defined on a single compact group
     $G$ (so that $G_n=G$ for all $n$), 
    in many cases it is  known that 
    the dimensions $d_n$ must be bounded.
    This was proved by Cecchini
      \cite{Ce} for $G$ a   Lie group
     and by Hutchinson \cite{Hut} for $G$ a profinite group. 
     Hutchinson's paper implies   the impossibility
     to have finite groups in Theorem \ref{t4} with  unbounded $d_n$'s. We should mention that  
      the latter reference
       (recently pointed out to the second author by  A. Fig\`a-Talamanca)
       already used Jordan's theorem, much like we do.
    
  Although Jordan gave no estimate for the growth of $f$,
 it was later proved by Blichfeldt, based on
 contributions notably by Bieberbach and Frobenius
 (see Remark \ref{bf} for details)  
 that this holds with $f(d)=O(d^{c(d/\log d)^2})$ and a fortiori
  with $f(d)=\exp{o(d^2)}$. The latter estimate is enough 
   to show that  Theorem \ref{t4} is void for  finite groups
   (see Corollary \ref{ceb}  for a precise statement).\\
More precisely,   if $d\ge 71$, any finite group 
 $\Gamma\subset U(d)$ has a normal Abelian subgroup of index at most $(d+1)!$,
 which  is sharp.
 This more recent   bound $(d+1)!$ is due to Collins \cite{collins},  \emph{but uses the
 classification of finite simple groups}. 
 The fact that $(d+1)!$ is sharp is easy: just consider the standard irreducible
 representation on $\C^{d+1}$ of the group of permutations of order $d+1$,
  restricted to the $d$-dimensional subspace $(1,\cdots,1)^\perp$,
  and note that the trivial subgroup  is the only normal Abelian subgroup and
  that  its  index is $(d+1)!$.
   Before Collins, a slightly weaker bound
 had been obtained by Boris Weisfeiler \cite{weisfeiler} (see
  \cite{collins} for details), before he disappeared in Chile, presumably murdered in early 1985.
 
  \section{On Jordan's theorem for amenable subgroups of $U(d)$}
 
 If one interprets Corollary \ref{c1} as a 
 quantitative density property, it is
 natural to wonder 
 about other properties of dense subgroups of $U(d)$.
In particular, since it is well known that for $d\ge 2$ 
  dense subgroups of $U(d)$ cannot be amenable,
  one may ask whether a  group satisfying (iv)
  (with $\alpha<\sqrt 2$ fixed and $d$ large enough)
  must be nonamenable (and a fortiori infinite !).
  Indeed, this turns out to be  true because
 Theorem \ref{b1} extends  to amenable subgroups of $U(d)$.
 The proof is a  reduction to the finite case, showing that any bound valid for finite subgroups of $U(d)$ will also be true for arbitrary amenable subgroups of $U(d)$.
This is due to the first author:
  \begin{pro} \label{tb} 
  Let $f(d)$ be a bound in Jordan's Theorem as above.  Any subgroup  $G\subset U(d)$   that is amenable as a discrete group
 has a normal Abelian subgroup of index at most $f(d)$.
 (In particular if $d\ge 71$ this holds with $f(d)=(d+1)!$ by \cite{collins}).
 \end{pro}

\begin{rem}
 Every Abelian subgroup of $U(d)$ can be simultaneously conjugated inside 
 the subgroup $D_d$ of diagonal matrices, so this implies that up to conjugating $G$ by a matrix in $U(d)$ we have $[G:G \cap D_d] \leq (d+1)!$.
\end{rem}
 
 \begin{proof}[Proof of Proposition \ref{tb}] Being amenable $G$ has a solvable subgroup of finite index, by the Tits alternative \cite{tits}. The closure of $G$ in the usual topology of $U(d)$ is a compact Lie subgroup with a solvable subgroup of finite index. Without loss of generality we may assume that $G$ is closed. Then the connected component of the identity $G^0$ is solvable. But solvable compact connected Lie groups are Abelian, isomorphic to $(\R/\Z)^k$ for some integer $k$. By a well-known fact due Borel-Serre \cite[Lemme 5.11]{borel-serre} and Platonov \cite[10.10]{wehrfritz}, there is a finite subgroup $H$ of $G$ such that $HG^0=G$. For an integer $n$, let $T_n:=\{t \in G^0 ; t^n=1\}$. This is a characteristic subgroup of $G^0$, which is isomorphic to $(\Z/n\Z)^k$. Hence it is normalized by $H$ and thus $H_n:=T_nH$ is a finite subgroup of $G$. We may apply Jordan's lemma with Collins' bound \cite{collins} to this $H_n$ and obtain an Abelian normal subgroup $A_n\leq H_n$ such that $[H_n:A_n]\leq f(d)$. In particular $[T_n:T_n\cap A_n]\leq f(d)$. If $n=p^m$ is a power of a prime $p$, then $T_n$ has no proper subgroup of index $<p$. So if $p>f(d)$, then $A_n$ contains $T_n$. Fix such a $p$. Note that the (increasing) union of all $T_{p^m}$, $m\geq 1$, is dense in $G^0$. This implies that the intersection of all $Z(T_{p^m})$ is $Z(G^0)$, where $Z(K):=\{t \in G; tk=kt \forall k\in K\}$ denotes the centralizer subgroup of $K$. 
 But   any decreasing sequence of compact subgroups of a given compact Lie group is stationary (``Noetherianity"),   so this intersection is finite, and hence there is $m\in \N$ such that $Z(T_{p^m})=Z(G^0)$. It follows that $A:=A_{p^m}G^0$ is an Abelian subgroup of $G$, which is normal and of index $$[G:A]=[HG^0:A_{p^m}G^0]=[H_{p^m}G^0:A_{p^m}G^0]\leq [H_{p^m}:A_{p^m}]\leq f(d).$$
\end{proof}

 \section{Consequence for the metric entropy}\label{pme}
 
 In this section, we show that
 Jordan's theorem (or Proposition \ref{tb}) with $f(d)=e^{o(d^2)}$ implies 
a non trivial property of the metric entropy
of any finite (or amenable) subgroup of $U(d)$.

Let $(T,\d)$ be any set $T$ equipped with a metric
or pseudo-metric $\d$.
Given a subset $S\subset T$ we denote
by $N(S,\d,\vp)$ the smallest number of a covering of $S$ by open balls of
$\d$-radius $\vp$. 

We will mainly consider the distances $\d_2$ and $\d_\infty$, corresponding to the Hilbert-Schmidt norm and the operator norm respectively, defined on $M_d$
as follows:
$$\forall u,v\in M_d \quad \d_2(u,v)=(d^{-1}\tr |u-v|^2 )^{1/2}$$
$$\forall u,v\in M_d \quad \d_\infty(u,v)=\|u-v\|.$$
Note  \begin{equation}\label{pe9}\d_2(u,v)\le \d_\infty(u,v).\end{equation}
 
\begin{lem}\label{el1}
For any $d$ 
and  any subgroup $G\subset U(d)$
containing an Abelian subgroup of index $k$
we have
 $$\forall \vp\in (0,2)\quad 
 N(G,\d_2,\vp) \le k (2\pi/\vp)^d.$$
\end{lem}
\begin{proof} 
Let $G=\cup_{j\le k} t_j \Gamma$ be the disjoint decomposition into cosets.
Then
$$ 
 N(G,\d_2,\vp) \le \sum\nl_{j\le k}  N(t_j \Gamma,\d_2,\vp) .$$
 Clearly $N(t_j \Gamma ,\d_2,\vp)=N( \Gamma,\d_2,\vp)$
 and since $\Gamma $ is Abelian 
 the matrices in $\Gamma$ are simultaneously diagonalizable
 so that we may assume that  $\Gamma$  is included in the set $
  D_d$ of all diagonal matrices with entries in $\T$. Thus by \eqref{pe9} we have
 $N( \Gamma,\d_2,\vp)\le N( D_d,\d_2,\vp)\le N( D_d,\d_\infty,\vp) \le (2\pi/\vp)^d$,
 from which the lemma follows.
 \end{proof}

\begin{rem}\label{pr3} Let $0<\vp<2$.
Let $A_\vp(d)$ be the smallest number $N$
such that any subgroup 
$G\subset U(d)$ satisfying
$$N(G,\d_2, \vp) > N$$
must be non-amenable as a discrete group, and let
$$H_\vp(d)=\log A_\vp(d).$$
By the preceding we have
$A_\vp(d)\le f(d) (2\pi/\vp)^d,$
and hence   
$H_\vp(d)\le \log f(d) + d\log (2\pi/\vp).$
Thus, assuming $d\ge 71$, the bound in Proposition \ref{tb}
implies a fortiori (by Stirling) $$H_\vp(d)\le d\log( d/e)+ d\log (2\pi/\vp).$$
We will   show in  \eqref{pe32} that this is   asymptotically sharp if we keep $\vp>0$
fixed and let
 $d\to \infty$. 
 
 But first we need to clarify the relationship between the various
ways to estimate the covering numbers  of groups with respect to a 
translation invariant 
 metric in the presence of a translation invariant
  probability (Haar) measure.\\
 Let $N'(G,\d_2, \vp) $ be the smallest number  
of a covering of $G$ by  
open balls of $\d_2$-radius $\vp$ with \emph{centers in $G$}.
It is easy to check that
$N( G,\d_2, \vp) \le N'(  G,\d_2, \vp)\le  N( G,\d_2, \vp/2)$ for any $\vp>0$.

We may consider the closure ${\bar G}\subset U(d)$ of $G$
equipped with its normalized Haar measure
 $m_{\bar G}$. Then by translation invariance,
 we have
 \begin{equation}\label{pe29} 1/m_{\bar G}(\{g\in \bar G\mid \d_2(g,1)<\vp\})\le N'(\bar G,\d_2, \vp)\le
1/m_{\bar G}(\{g\in \bar G\mid \d_2(g,1)<\vp/2\}).\end{equation}
Obviously, we have $N'(\bar G,\d_2, \vp_1) \le N'(  G,\d_2, \vp) \le N'(\bar G,\d_2, \vp)$
for any $\vp_1>\vp$.\\
Thus (say) $m_{\bar G}(\{g\in \bar G\mid \d_2(g,1)<3\vp\})<  1/A_\vp(d)$
implies that $G$ is non-amenable.

To be more concrete, if we set, say, $\vp=1/30$,   there is $c'>0$ such that for all $d$ large enough if 
$$\log \frac{1}{m_{\bar G}(\{g\in \bar G\mid \d_2(g,1)<1/10\}) }  \ge c' d\log d $$
then  $G$ is not amenable.
We will now show that this is  asymptotically sharp.
\end{rem}
\begin{rem}[A case study]\label{pr1}
Let $\cl G\subset U(d)$ (actually $\cl G\subset O(d)$) be the finite subgroup
formed of all the matrices of the form
$$u=\sum\nl_1^d \vp_i e_{i,\sigma(i)} $$ 
where $(\vp_i)\in \{-1,1\}^d$ and $\sigma$ is in  the symmetric group $ S(d)$.
The group $\cl G$ is isomorphic to   the semidirect product
${\{-1,1\}^d} \rtimes S(d)$.
Then
$$\tr(u)=  \sum\nl_{ i\in {\rm Fix}(\sigma)}  \vp_i
$$
$$\d_2(u,1)^2= 2(d-\tr(u))= 2 \sum\nl_{ i\in {\rm Fix}(\sigma)} (1-\vp_i)  +2 (d-  |  {\rm Fix}(\sigma) |) $$
where ${\rm Fix}(\sigma)=\{i\mid \sigma(i)=i\}$. 
For any $\vp>0$ we have
\begin{equation}\label{pe30}\{u\in {\cl G}\mid  \d_2(u,1)<\vp\}=\{u\in {\cl G}\mid  \tr(u)> d(1-\vp^2/2)\}.
\end{equation}
Let $X_j $ be the number of permutations in $S(d)$ with exactly $j$ fixed points. 
Then for any $0\le k< d$
\begin{equation}\label{pe15} m_{\cl G}(\{u\in {\cl G}\mid  \tr(u)>  k\}
=(d!)^{-1}\sum\nl_{j>k} X_j \P(\{S_j>k\})\end{equation}
where $S_j=\vp_1+\cdots+\vp_j$ is the sum of $j$ independent 
(uniformly distributed) choices of signs, and 
\eqref{pe15} is $=0$ when $k\ge d$. Thus we have  for any $0\le k< d$ (note that $X_d=1$
and $\P(\{S_d>k\})\ge 2^{-d}   $)
$$ (d!)^{-1}  2^{-d} \le m_{\cl G}(\{u\in {\cl G}\mid  \tr(u)>  k\} \le   (d!)^{-1}\sum\nl_{j>k} X_j . $$

It is easy to see
that $X_j= {{d} \choose{ j}} D(d-j)$ where $D(n)$ denotes the number
of derangements of an $n$-element set, i.e. the number of permutations
without fixed point in $S(n)$. It is well known  
(see e.g. \cite[p. 67]{Sta})  
that $D(n)$ is of order $n!/e$  when $n\to \infty$, 
and more precisely: for any $n\ge 1$    (note $D(1)=0$)
 \begin{equation}\label{pe11} D(n)=n! (1-\frac{1}{1!}+\frac{1}{2!}-\frac{1}{3!}+\cdots+(-1)^n\frac{1}{n!}).\end{equation}
This shows $n! \ge D(n)\ge n! /3$ for all $n>1$.
Contenting ourselves (for the moment) with the obvious
bound $D(d-j)\le (d-j)!$ we find
 \begin{equation}\label{pe12}(d!)^{-1} X_j\le (j!)^{-1},\end{equation} and hence
$(d!)^{-1}\sum\nl_{j>k} X_j \le e {(k+1)} !^{-1} $.
By Stirling's formula
 $(d!)^{-1}\sum\nl_{j>k} X_j \le e (e/(k+1))^{k+1}$. Therefore  for any $0\le k< d$ 
 \begin{equation}\label{pe31} (d!)^{-1}  2^{-d} \le m_{\cl G}(\{u\in {\cl G}\mid  \tr(u)>  k\} \le  e (e/(k+1))^{k+1} . \end{equation}
 Recalling \eqref{pe29} and \eqref{pe30} and choosing $k =[d(1-\vp^2/2)]$ we find
 $$\log N(\cl G,\d_2,\vp/2) \ge  \log N'(\cl G,\d_2,\vp) \ge   ({k+1} )\log ((k+1))-k -2,$$
 from which we deduce, for any $0<\vp <\sqrt 2$  
\begin{equation}\label{pe32}H_{\vp/2} (d) \ge  (1-\vp^2/2) d \log (d/e)   -c_3,\end{equation}
 where $c_3$ is a fixed constant independent of $d$.
\end{rem}
 
\begin{rem}
Similarly, assuming $G\subset U(d)$ amenable,
let $\varphi$ be an invariant mean on $G$. 
Since both distance and mean are translation invariant
it is easy to check that
$$1/\varphi(\{g\in G\mid \d_2(g,1)<\vp\})\le N'(G,\d_2, \vp)\le
1/\varphi(\{g\in G\mid \d_2(g,1)<\vp/2\}).$$
By the preceding reasoning $\varphi(\{g\in \bar G\mid \d_2(g,1)<3\vp\})<  1/A_\vp(d)$
would imply that $G$ is not amenable. Therefore we must have
$$\varphi(\{g\in \bar G\mid \d_2(g,1)<3\vp\})\ge  1/A_\vp(d) \ge ( f(d) (2\pi/\vp)^d)^{-1},$$
and hence, say taking $\vp=1/9$ 
$$\varphi(\{g\in \bar G\mid \d_2(g,1)<1/3\})\ge (18\pi /d)^d /(d+1).$$
\end{rem}

\section{Subgaussian variables}\label{sp1}

To conform with a commonly used notation, we set
$$\forall x\in \R_+\quad \psi_2(x)=e^{x^2} -1.$$

Given a measure space $(T,m)$ we denote by $L_{\psi_2}(T,m)$ the (Orlicz)
space   formed of all  the measurable complex valued functions $F$
for which  there is $c>0$
such that $\int \psi_2(|F|/c)   dm <\infty$. We denote
  \begin{equation}\label{ep66}\|F\|_{\psi_2}=\inf\{c>0\mid \int \psi_2(|F|/c)  dm\le \psi_2(1)  \}.\end{equation}
  
When $F$ is real valued with $\int Fdm=0$ and
$m(T)=1$    it is not hard to show
that   $\|F\|_{\psi_2}$
 is equivalent to the smallest constant $C$ such that
 $$\forall t\in \R\quad 
  \int \exp{ ( t  F -Ct^2/2)} dm \le 1.$$ 
  Since the equality case
    for  $C=1$ characterizes the standard Gaussian variables,
    this explains why we view (ii) as a subgaussian estimate.
    
    Using the Taylor expansion of the exponential function it is easy to show
    that $\|F\|_{\psi_2}$
 is equivalent (with absolute equivalence constants)
 to $F\mapsto \sup\nl_{p\ge 2} \|F\|_{p}/\sqrt p$. More precisely,
 we can restrict if we wish to even integers: there is a constant $\lambda>0$
 such that for any complex valued measurable  $F$
 we have
  \begin{equation}\label{se3}  \lambda^{-1}     \|F\|_{\psi_2} \le    \sup\nl_{n\in \N_*} \|F\|_{2n}/\sqrt {2n}  \le \lambda\|F\|_{\psi_2}  .\end{equation}
 Consider the case when $T$ is a compact group $G$
 with  $m=m_G$ and let $F=\chi_\pi$ be the character of some $\pi\in \hat G$.
 Then, for any $i,j\ge 0$,  the unitary representation $ \pi^{\otimes i} \otimes \bar{\pi}^{\otimes j} $
 admits a decomposition into irreducibles that we may write as:
  \begin{equation}\label{pe8}\pi^{\otimes i} \otimes \bar{\pi}^{\otimes j} =\oplus_{\sigma\in \hat G} m_\pi (i,j;\sigma) \sigma,\end{equation}
  where the integer $m_\pi (i,j;\sigma)$ is the multiplicity (possibly $=0$)
  of $\sigma$ in $ \pi^{\otimes i} \otimes \bar{\pi}^{\otimes j} $.
  Taking the $L_2$-norm of the trace of both sides of \eqref{pe8} we find
  $$\int |\chi_\pi|^{2(i+j)} dm=\sum\nl_{\sigma\in \hat G} m_\pi (i,j;\sigma)^2.$$
  Therefore, the condition 
  $$ \sup\nl_{n\in \N_*} \|\chi_\pi\|_{2n}/\sqrt {2n}  \le C$$
can be reformulated ``arithmetically" as saying that for any $i,j\ge 0$
(or merely for all $i\ge 1$ with $j=0$)
$$\sum\nl_{\sigma\in \hat G} m_\pi (i,j;\sigma)^2  \le C^{2(i+j)} ({2(i+j)} )^{i+j}.$$

\section{Random Fourier series} 

We   describe in this section the connection of 
Theorem \ref{t4} and Corollary \ref{c1}
 to  Gaussian random Fourier series
in the style of \cite{MaPi}. More recent
information on general Gaussian random processes can be found
in \cite{Ta2}.

We denote by $g_d$ a random $d\times d$-matrix with entries
 $\{g_{d} (i,j)\mid 1\le i,j\le d\}$ forming an i.i.d. family
 of complex valued Gaussian variables
 with $\EE|g_{d}(i,j)|^2=1/d$, on a suitable probability space
$(\Omega, \P)$.

Let $G$ a compact group $G$ and let $(a_\sigma )_{\sigma\in \hat G}$
be a family of ``Fourier coefficients", i.e. assuming that 
$\sigma$ takes its values in $U(d_\sigma)$ we assume
that $a_\sigma\in M_{d_\sigma}$. We also assume
that $\sum d_\sigma \tr(|a_\sigma|^2) <\infty$.
 The  associated random
Fourier series is 
 the random process $(S_t)_{t\in G}$  
defined by
\begin{equation}\label{pe16} S_t(\omega)=\sum\nl_{\sigma\in \hat G} d_\sigma \tr (a_\sigma g_{d_\sigma}(\omega) \sigma(t)),\end{equation}
where the family of random matrices $(g_{d_\sigma} )_{\sigma\in \hat G}$ is an independent one. 
We associate to it the
pseudo-distance defined on $G$ by $\d_S(s,t)=\|S_s-S_t\|_2$.  The main results in \cite{MaPi}
show  that the Dudley-Fernique entropy
condition
$$\int_0^\infty (\log N(G,\d_S,\vp) )^{1/2} d\vp<\infty$$
 that was known to characterize the a.s. boundedness of $(S_t)$
is also equivalent to the a.s. boundedness of  random
Fourier series associated to more general randomizations than the Gaussian one. In particular,
 the same characterization holds  
 for independent unitary matrices
uniformly distributed over $\prod\nl_{\sigma\in \hat G} U(d_\sigma)$ in place of
 $(g_{d_\sigma} )_{\sigma\in \hat G}$.  In fact these results do not require the irreducibility of the  
  $\sigma$'s, as long as one uses the metric entropy associated to  
   $\d_S$. If one removes the 
  irreducibility assumption, even the case of $S_t$ reduced to a single sum
  $S_t=\tr(a g_{d_\pi}\pi(t))$ with $a\in M_{d_\pi}$ is non trivial, and actually
  it can be argued (by decomposing $\pi$ into irreducible components) that this case is equivalent to  the one in \eqref{pe16}.
  In this paper, we concentrate  
  on the even more special case
  when $a$ is the identity matrix.
  
  Let $G$ be any group
  and let $\pi:\ G\to U(d)$ be a representation. 
  We will estimate the random variable $Z_\pi$ defined on $(\Omega, \P)$ by
  $$Z_\pi(\omega)
  = \sup_{t\in G}| \tr (g_d(\omega) \pi(t) )|.
  $$
For our considerations,  it will be essentially equivalent
  to replace it by the variable
  $$u\mapsto  \sup_{t\in G}| \tr (u \pi(t) )|,
  $$
defined when $u$ is chosen uniformly in $ U(d)$.\\
  We associate to $\pi$ the (pseudo-)distance $\d^\pi$ 
  defined on $G$ by
  $$\d^\pi(s,t)=  (d^{-1} \tr|\pi(s)-\pi(t)|^2)^{1/2}.$$
  We will repeatedly use the observation that
  \begin{equation}\label{pe10}
  \d^\pi(s,t)^2=    2(1  -d^{-1}\chi_\pi(s^{-1}t) ).\end{equation}
  Let $\vp>0$.
 We denote by $N(\pi,\vp)$ the smallest number of a covering of $G$ by
 open balls  of radius
$\vp$ for the metric $\d^\pi$.
 We then 
introduce the so-called metric entropy integral  
  $$ {\cl I}(\pi)= \int_0^2 (\log N(\pi,\vp))^{1/2}d\vp.$$
 Note $\log N(\pi,\vp)=0$ for all $\vp>2$ since the diameter of $G$ is at most 2.
 
In the present very particular situation 
 the Dudley-Fernique theorem for Gaussian random Fourier series (see \cite{MaPi}),
   says that there are numerical positive constants $b_1,b_2$
 such that for any $G$, $\pi$ and $d$  
  \begin{equation}\label{pe3} b_1 {\cl I}(\pi) \le \E  Z_\pi  \le
 b_2 {\cl I}(\pi).\end{equation}
 
 By elementary arguments  (based on the translation invariance
       both of the metric ${\d}^\pi$ and the measure $m_G$)
       we have  (as in \eqref{pe29}) for any $\pi$
   \begin{equation}\label{ab} {m_G(\{t\mid {\d}^\pi(t,1)<\vp\})}^{-1}\le N(\pi,\vp)\le {m_G(\{t\mid {\d}^\pi(\pi,1)<\vp/2\})}^{-1},\end{equation}
   so that ${\cl I}(\pi) $
is equivalent to
${\cl I}'(\pi) =\int_0^2 (-\log {m_G(\{t\mid {\d}^\pi(t,1)<\vp\})}  )^{1/2}d\vp$.
 
By the comparison arguments from \cite{MaPi} we also have
for suitable constants $b_1,b_2$
  \begin{equation}\label{pe3'} b_1 {\cl I}(\pi) \le \int_{U(d)} \sup\nolimits_{t\in G} | \tr ({u}  \pi(t)) | m_{U(d)}(d{u})  \le
 b_2 {\cl I}(\pi).\end{equation}
 A fortiori, this shows that \\
 \centerline{$M_u=\int_{U(d)} \sup\nolimits_{t\in G} | \tr ({u}  \pi(t)) | m_{U(d)}(d{u})  
\text{  and  }  M_g=\E\sup\nolimits_{t\in G} | \tr (g_d  \pi(t)) | \ (=\E  Z_\pi)$}
are equivalent.\\
Actually, in the present situation, the latter equivalence 
can be proved directly very easily, using the matricial
version of the ``contraction principle"  in \cite[p. 82]{MaPi}. We briefly indicate
the argument: one direction uses the fact that the polar decomposition
 of $g_d$ is such that $g_d$ (with respect to $\P$) has the same distribution
 as the variable $u|g_d|$  with respect to $m_{U(d)}\times \P$ on the product $U(d)\times \Omega$. This implies that  $u\E|g_d|$ can be obtained
 from  $g_d$ by the action of a conditional expectation. Since $\E |g_d|= b_d I$
 with $b_d\ge b$ for some  numerical constant $b>0$,
 this gives us $b M_u\le   M_g$. To prove the converse,
 we note (``contraction principle") that   a convex function on $M_d$ is maximized on the unit ball at an extreme point, i.e. at a matrix in $U(d)$, and so for any fixed $\omega$, we have
 $$ \int_{U(d)} \sup\nolimits_{t\in G} | \tr ({u} g_d(\omega) \pi(t)) | m_{U(d)}(d{u})
 \le \|g_d(\omega)\| \int_{U(d)} \sup\nolimits_{t\in G} | \tr ({u}   \pi(t)) | m_{U(d)}(d{u}) $$
 and hence after integration in $\omega$ with respect to $\P$
$$ \int_{U(d)}  \sup\nolimits_{t\in G} | \tr ({u} g_d(\omega) \pi(t)) | m_{U(d)}(d{u})d\P(\omega) 
 \le \E\|g_d \| \int_{U(d)} \sup\nolimits_{t\in G} | \tr ({u}   \pi(t)) | m_{U(d)}(d{u}). $$
 Since, as is well known,  $ \E\|g_d \|  $ remains bounded by a constant $b'$ when $d\to \infty$
(see e.g. \cite[p. 78]{MaPi}) this implies
the converse inequality
$ M_g\le   b' M_u$.
 
 Since $N(\pi,\epsilon)$ is a non-increasing function of $\epsilon$, note also the elementary minoration
 \begin{equation}\label{sud}
 \sup\nl_{0<\vp\le 2 } \vp (\log N(\pi,\vp))^{1/2} \le {\cl I}(\pi) ,\end{equation} 
 which, by \eqref{pe3'}, gives us the lower bound
\begin{equation}\label{sud'}  b_1\sup\nl_{0<\vp\le 2 } \vp (\log N(\pi,\vp))^{1/2} \le \int_{U(d)} \sup\nolimits_{t\in G} | \tr ({u}   \pi(t)) | m_{U(d)}(d{u}). 
\end{equation}
In the Gaussian case, 
we also have $b_1\sup\nl_{0<\vp\le 2 } \vp (\log N(\pi,\vp))^{1/2} \le M_g$. The latter is known 
 as the Sudakov minoration (see e.g. \cite[p. 69]{Piv} or \cite[p. 80]{LeTa}). While the preceding 2-sided bound \eqref{pe3}
 requires the translation invariance
 of the distance (or the stationarity of the associated Gaussian process),
 Sudakov's lower bound holds  for general Gaussian processes.

 \section{Proofs}\label{pfs}
 
 We first indicate where 
 the proof of Theorem \ref{t4} can be found.
  By \cite[Cor. 5.4]{Pi2} (see also \cite{Piri})
    (i) and (ii) in Theorem \ref{t4} are equivalent. Moreover, by
     \cite[Prop. 5.3]{Pi2} 
     they are equivalent
    to (iii). Note that complete details for this can be found in \cite{Piri}
    (together with a correction to another assertion in  \cite{Pi2}).
    As for the equivalence between (ii) and (iii), a more precise 
    two sided inequality holds for the corresponding best possible constants:
    \begin{lem}\label{fl1}
    Let $G$ be a compact group, $d\ge 1$ and  $\pi:\ G \to U(d)$  an irreducible unitary
    representation. We set 
    $$C_2(\pi)= \|\chi_\pi\|_{\psi_2} \text{  and  } C_3(\pi)=\frac{d}{\int_{U(d)} \sup\nolimits_{t\in G} | \tr ({u}   \pi(t)) | m_{U(d)}(d{u})}.$$
    There is a numerical constant $K>0$
  such that for any
       $G,d,\pi$
       $$K^{-1} C_3(\pi)  \le C_2(\pi) \le K C_3(\pi) .$$ 
    \end{lem}
        \begin{proof} 
Firstly we will show    $K^{-1} C_3(\pi)  \le C_2(\pi)$.
Note that
  $1= (\tr(|u|^2/d))^{1/2} = \|\tr(u\pi)\|_2 \le \|\tr(u\pi)\|_\infty$  for   any fixed $u\in U(d)$, therefore
  $C_3(\pi) \le d.$\\
        Let $0<\vp<\sqrt 2$. Let $C=\|\chi_\pi\|_{\psi_2}$.
   Then  $\|\Re(\chi_\pi)\|_{\psi_2}  \le C$ and hence
   $$ \exp({(d/C)^2(1-\vp^2/2)^2})  m_G(\{ \Re(\chi) > d(1-\vp^2/2)\}) \le  e.$$
  Taking the square root of the log, we find
  $$(d/C) (1-\vp^2/2) \le 1 + \left(\log 1/m_G(\{ \Re(\chi) > d(1-\vp^2/2)\})\right)^{1/2}$$
  and hence using   \eqref{ab}
  $$\le 1 +  \left(\log N(\pi,\vp) \right)^{1/2}.$$
  By   \eqref{sud'} this is
  $$\le 1 +(b_1 \vp)^{-1} \int_{U(d)} \sup\nolimits_{t\in G} | \tr ({u}   \pi(t)) | m_{U(d)}(d{u}) 
\le 1 + (b_1 \vp)^{-1} d/ C_3(\pi).$$ 
Thus we obtain 
$$(d/C_2(\pi)) (1-\vp^2/2) \le 1+ (b_1 \vp)^{-1} d/ C_3(\pi),$$
or equivalently (multiplying by $C_2(\pi)C_3(\pi)/d$)
$$ C_3(\pi) (1-\vp^2/2) \le C_2(\pi)C_3(\pi)/d +(b_1 \vp)^{-1} C_2(\pi) $$
and since $C_3(\pi) \le  d$, 
  choosing say, $\vp=  1$, we obtain
$C_3(\pi) \le  K  C_2(\pi) $
with $K= 2  ( 1+b_1  ^{-1}) $.

We now turn to the converse direction.  
Let $${\Phi}(a)= \int_{U(d)}
 \sup\nolimits_{g\in G} | \tr ({u} a \pi(g)) | m_{U(d)}(d{u}).$$
  We first claim   that for any matrix $a\in M_d$
 \begin{equation}\label{fe1}\tr|a|\le C_3(\pi) {\Phi}(a).\end{equation}
This follows
from a simple averaging argument.  Indeed let $a =u|a|$ be
 the polar decomposition   $(|a|^2=a^*a)$,
then ${\Phi}(a)={\Phi}(|a|)$ so that 
to show \eqref{fe1}
it suffices to show that
$   \tr(a)\le C_3(\pi) {\Phi}(a)$ 
for all  $a$ in $M_d$. Then for any $s\in G$
we have
${\Phi}(a)= {\Phi}( \pi(s) a)$ (by translation in variance over $U(d)$)
and
${\Phi}(a)= {\Phi}(  a\pi(s^{-1}) )$ (by translation in variance over $G$).
By convexity
$${\Phi}(a)= {\Phi}( \pi(s) a\pi(s^{-1}) )\ge {\Phi}( \int \pi(s) a\pi(s^{-1}) m(ds))=  
{\Phi}(\tr(a) I/d)= \tr(a) {\Phi}(I)/d$$
and this gives us \eqref{fe1}, since ${\Phi}(I)= d/C_3(\pi)$ by definition of $C_3(\pi)$.

We now interpret \eqref{fe1} as saying that the norm of a natural inclusion
$J$ between two normed
spaces is at most $C_3(\pi)$: we write $\|J:\ X\to M_d^*\|\le C_3(\pi)$ (the space $X$
is $M_d$ equipped with the norm ${\Phi}$).
By duality
the inequality \eqref{fe1} means that 
 \begin{equation}\label{fe2} \|J^*:\ M_d\to  X^*\|\le C_3(\pi).\end{equation}
By the duality theorem in \cite[p. 116]{MaPi}
  the dual of $X$
can be identified (up to a fixed isomorphic constant)
with the space of Fourier multipliers
from $L_2(G) \to L_{\psi_2}(G)$.
It follows that for any $b\in M_d$
 \begin{equation}\label{fe3} \|J^*(b)\|_{X^*}=  \sup\{ |\tr(ba)|\mid {\Phi}(a)\le 1 \}\simeq \sup\{ \|\tr (ba\pi )\|_{\psi_2}\mid  \|\tr (a\pi )\|_{2} \le 1 \} .\end{equation}
 Since \eqref{fe2} implies
$\|J^*(b)\|_{X^*} \le C_3(\pi) \|b\|_{M_d}$, 
  taking $b=a=I$ we obtain from \eqref{fe3} the announced bound
$$C_2(\pi)=\|\tr ( \pi )\|_{\psi_2} \le  K C_3(\pi).$$
 \end{proof}

\begin{rem} More generally if we work with a subset $\Lambda \subset \hat G$
the duality theorem says that 
the best constant in
  \begin{equation}\label{fe4}\forall a_\pi 
\quad\sum\nl_\Lambda d_\pi \tr |a_\pi| \le C  \E \|\sum\nl_\Lambda d_\pi \tr ( u_\pi  a_\pi \pi)\|_\infty \end{equation}
and
  \begin{equation}\label{fe5} \forall a_\pi 
\quad \|\sum\nl_\Lambda d_\pi \tr (  a_\pi  \pi)\|_{\psi_2} 
\le C \|\sum\nl_\Lambda d_\pi \tr (a_\pi  \pi)\|_{2} \end{equation}
are equivalent.

Moreover  by the same averaging argument
(based on ireducibility of the $\pi$'s)
the best constant in  \eqref{fe4}
is the same if we restrict \eqref{fe4} to the case when the $a_\pi$'s are scalar matrices.
\end{rem}
   \begin{rem} Let $C_1(\pi)$ be the best constant associated to (i)
   in Theorem \ref{t4}. More precisely (this is the Sidon constant of $\{\pi\}$ in the sense of \S \ref{sid}), we define
   $$C_1(\pi)=\sup\{ \tr|a|\mid a\in M_d,\    \sup\nolimits_{g\in G} |\tr (a\pi(g))|\le 1\}.$$
   Obviously $C_3(\pi)   \le C_1(\pi)$. In the converse direction,
   the best known estimate seems to be 
  \begin{equation}\label{fe6}  C_1(\pi)\le K' C_3(\pi)^2 \log(1+C_3(\pi)),\end{equation}
   for some numerical constant K'.
   To check this we first  invoke again the duality
   theorem  in \cite[p. 116]{MaPi}. This implies
   that for any $a\in M_d$ we have
   $$ \|\tr(\pi)\|_{\psi_2} \le K'' C_3(\pi) (\tr(|a|^2))^{1/2},$$
   for some numerical constant $K''$.  
   Then \eqref{fe6}
    can be deduced from the proof of 
    \cite[Th. 3.7]{Pi3} if one takes into account the logarithmic growth 
    described   \cite[Rem. 1.16]{Pi3}.
\end{rem}
   The next statement follows from Theorem  \ref{t4} by the same simple argument 
 already used in the Abelian case in \cite{Pi5}. 
   \begin{cor}\label{cp2} The properties  in Theorem \ref{t4} are equivalent to
   the following ones:
   \item{(v)} There are  numbers $\beta >0$ and $\alpha>0$ such that
   for any $n$ there is a subset $A_n\subset G_n$
   with $|A_n|\ge e^{-1} e^{\alpha  d_n^2}$
   such that 
   $$\forall s\not=t\in A_n\quad
   \|\pi_n(s)-\pi_n(t)\|  > \beta .$$
   
    \item{(v)'} There are  numbers $\beta' >0$ and $\alpha'>0$ such that
   for any $n$ there is a subset $A'_n\subset G_n$
   with $|A'_n|\ge  e^{-1} e^{\alpha'  d_n^2}$
   such that 
   $$\forall s\not=t\in A'_n\quad
   (d_n^{-1/2} \tr|\pi_n(s)-\pi_n(t)|^2)^{1/2}  > \beta'  .$$
      \end{cor}
       \begin{proof} 
       We will first show that (i) and (ii) in Theorem \ref{t4} are equivalent
       to {(v)'}. 
     This is an easy consequence of the subgaussian estimate (ii)
      and of \eqref{ab}. Indeed, let $\pi=\pi_n$, $G=G_n$.
  Recall 
${\d}^\pi(t,1)^2= 2  ( 1-d_\pi^{-1} \Re \chi_\pi(t))$. Therefore
(ii) implies assuming $\vp<\sqrt2$
 \begin{equation}\label{ep6}  {m_G(\{t\mid {\d}^{\pi}(t,1)<\vp   \})}=
{m_G(\{t\mid   \Re \chi_\pi(t)>d_\pi(1-\vp^2/2)\})}\le e e^{ -\gamma d_\pi^2}\end{equation}
where $\gamma= \beta (1-\vp^2/2)^2 $. From this follows by \eqref{ab}
$$N({\pi} ,\vp   )\ge 
{m_G(\{t\mid   \Re \chi_\pi(t)>d_\pi(1-\vp^2/2)\})}^{-1} \ge 
e^{-1} e^{ \gamma d_\pi^2}.$$
 Let $A\subset G$      be a maximal subset of
 points such that 
 $$\forall s\not=t\in A\quad \d^\pi(s,t)=  (d^{-1} \tr|\pi(s)-\pi(t)|^2)^{1/2} > \vp  /2.$$
 Clearly, $|A|\ge N({\pi} ,\vp   )$.
 Therefore, (v)' follows with  $\beta'=\vp/2$ (which can be any number $<1/\sqrt 2$),
 and $\alpha'=\gamma^{1/2}  $.
 This shows (ii) implies (v)'. Conversely, assume  (v)'.
 We will show that (iii) holds.
 Indeed, (v)' implies a lower bound
 $N(\pi,  \beta'/2) \ge |A_n|\ge  e^{-1} e^{\alpha'  d_n^2}$,
 and plugging this into \eqref{pe3'}  and
 \eqref{sud}
 we immediately derive (iii).
 
To complete the proof
we will show that (v) and (v)'
are equivalent.  Clearly (v)' implies (v). For the converse, we will use the following non-commutative analogue of a result  from approximation theory (see e.g. \cite{Ca}).
\begin{slem} Let $B_d =\{x\in M_d\mid  (d ^{-1/2}\tr|x|^2)^{1/2} \le 1\}$.
For any $\xi>0$ there is a constant $r_\xi$ such that,  
for any $d$, we have
 $$N(B_d , \d_\infty,r_\xi)\le  \exp{(\xi d^2)}.$$
\end{slem} 
To prove the sublemma,   given subsets $K_1,K_2$ of $M_d$ let us denote
for $\vp>0$
by $N(K_1,\vp K_2)$ the smallest number 
of a covering of $K_1$ by translates of $\vp K_2$.
If $K_3$ is another set,   obviously we note for later use that for any $r>0$
 \begin{equation}\label{ep5} 
 N(K_1,\vp r K_2)\le N(K_1,\vp   K_3) N( \vp   K_3 ,\vp r K_2)
 =N(K_1,\vp   K_3) N(     K_3 ,   r K_2)
  .\end{equation}
Let ${\cl B}_d$ be the unit ball of $M_d$ (equipped with the operator norm).
Note ${\cl B}_d \subset {B}_d$. The sublemma
is clearly equivalent to the claim that for any $\xi>0$
there is   $r_\xi$  such that
$N(B_d, r_\xi {\cl B}_d) \le \exp{(\xi d^2)}$.
There are many possible proofs of the latter.
We choose one for which we have the references at hand.
By \cite[Cor. 5.12 p. 80]{Piv} (up to a change of notation)
there is an absolute constant $C$ such that
$$\sup\nl_{\vp>0} \vp (\log N(B_d, \vp {\cl B}_d) 
\le C d\  \E\|g_d \| .$$
Since, as we already mentioned,  $ \E\|g_d \|  $ remains bounded when $d\to \infty$
(see e.g. \cite[p. 78]{MaPi}) we may modify the absolute  constant $C$
so that 
$$\sup\nl_{\vp>0} \vp (\log N(B_d, \vp {\cl B}_d) 
\le C d  .$$
Then, choosing $\vp=C/\xi^{1/2}$, we find as announced
$N(B_d, r_\xi {\cl B}_d) \le \exp{(\xi d^2)}$
with $r_\xi= C/\xi^{1/2}$, completing the proof of the sublemma.\\
We now show that  (v) $\Rightarrow$ (v)'.
 Assume (v). 
 Again we set $G=G_n, \pi=\pi_n$ and $d=d_\pi$.
 Then (v) implies
 $  |A_n| \le N(\pi(G), (\beta/2) {\cl B}_d) $.
By \eqref{ep5} we have for any $r>0$
$$|A_n| \le N(\pi(G), (\beta/2r) { B}_d) 
N({ B}_d,  r {\cl B}_d)  .$$
Choosing $r=r_\xi$ gives us
$$ {|A_n|}\exp{-(\xi d^2)} \le N(\pi(G), (\beta/2r_\xi) { B}_d) ,$$
then choosing (say) $\xi= \alpha/2$
we obtain  
$ e^{-1} e^{\alpha  d^2/2} \le N(\pi(G), (\beta/2r_\xi) { B}_d) .$
From this considering as usual a maximal set of points
$A'\subset G$ such that
$(d ^{-1/2}\tr|\pi(s)-\pi(t)|^2)^{1/2}
 > \beta/4r_\xi$  for all  $s\not=t\in A'$,
we have necessarily $|A'|\ge N(\pi(G), (\beta/2r_\xi) { B}_d)$.
Thus we obtain (v)' with $\beta'= \beta/4r_{\alpha/2}$.
        \end{proof}
        
         \begin{lem}\label{tb2} Let $G$ be a group with an Abelian subgroup $\Gamma$
  of index $k<\infty$.
  Let $\pi:\ G \to U(d)$ be a unitary representation.
  Then
 \begin{equation}\label{e5}\EE\sup\nl_{t\in G}|\tr(g_d\pi(t))|\le  \sqrt{2\log k}+ \sqrt{d}.\end{equation}
 \end{lem}
   \begin{proof}
   Up to an extra constant factor,
   this can be easily derived from 
   Lemma \ref{el1} and \eqref{pe3} by plugging the estimate
of Lemma   \ref{el1}    into  the upper bound of \eqref{pe3}.
   We give a direct proof for the convenience of the reader.
    Let $G=\cup_{j\le k} t_j \Gamma$ be the disjoint decomposition into cosets.
   Let $$Y_j=\sup_{t\in t_j\Gamma}|\tr(g_d \pi(t))|=
   \sup_{\gamma\in \Gamma} |\tr(g_d \pi(t_j) \pi(\gamma)|.$$
   Then $$\sup_{t\in G}|\tr(g_d\pi(t))|=\sup_j Y_j.$$
   Since $\{\pi(\gamma)\mid \gamma \in \Gamma\}$ are commuting unitary matrices, they are simultaneously diagonalizable, i.e. $\exists V\in U(d)$
   such that $\pi (\gamma)= V D(\gamma)V^{-1}$ where $D(\gamma)$
   is a diagonal matrix. Then
   $\tr(g_d \pi(t_j) \pi(\gamma)=\tr(V^{-1}g_d  \pi(t_j) V D(\gamma))$.
   Moreover, $V^{-1}g_d  \pi(t_j) V \overset{\rm dist} {=}  g_d$.
   Therefore, $Y_j \overset{\rm dist} {=} W(g_d)$
   where
   $$W(g_d)=   \sup_{\gamma\in \Gamma} |\tr(g_d D(\gamma))|=\sup_{\gamma\in \Gamma} |\sum g_d(i,i) D_{i,i}(\gamma)| \le \sum |g_d(i,i)|.$$
The function $[x_d(i,j)]\mapsto W(x_d/\sqrt {d})$ is Lipschitz
   on $\C^{d^2}=\R^{2d^2}$ (equipped with the Euclidean norm) with distortion $ \le 1$. Therefore (see \cite[p. 181]{Pip} 
   or  \cite{Le,LeTa})
   for any $\lambda \in \RR$ 
   $$\EE \exp{ \lambda (W-\EE W)}\le   \exp{(\lambda^2/2)}.$$
   Now
   $$\EE \exp{\sup_j  \lambda (Y_j-\EE Y_j)} \le \sum_j \EE \exp{ \lambda (Y_j-\EE Y_j)}\le k   \exp{(\lambda^2/2)}.$$
   A fortiori by convexity for any $\lambda >0$
   $$ \exp{ ( \lambda\EE\sup_j (Y_j-\EE Y_j))} \le k   \exp{\lambda^2/2}.$$
   Let $R=\EE\sup_j (Y_j-\EE Y_j)$.
   We have $ \exp{  \lambda R-\lambda^2/2} \le  k  $
   and hence (take $\lambda=R$)
   $$R \le \sqrt{ 2 \log k  }.$$
   Clearly
    $$ \EE\sup_j Y_j \le \EE\sup_j (Y_j-\EE Y_j)+ \sup_j \EE Y_j\le
    R+ \EE\sum |g_d(i,i)|\le R+d\EE|g_d(1,1)|\le R+\sqrt d.$$
    From this the announced result follows.
   \end{proof} 
     
     \begin{lem}\label{lp2} In the situation of Lemma \ref{tb2}, we have
     $$\forall \vp<\sqrt 2\quad 
     m_G(\{\Re\chi_\pi > d(1-\vp^2/2)\}) \ge k^{-1} (\vp/2\pi)^d,$$
     and
    \begin{equation}\label{e6}  \|\chi_\pi\|_{\psi_2}\ge 
    c'  \min\{\sqrt{d},  \frac{d}{\sqrt{\log k} } \}, \end{equation}
     where $c'>0$ is a numerical constant.
    \end{lem}
    \begin{proof}  Going back to the definitions, 
 we find that $ N(\pi,\vp)$ is essentially the same as $N(\pi(G), \d_2,\vp)$,
 but using only balls   centered
 in $\pi(G)$.
Thus   $N(\pi,2\vp)\le N(\pi(G), \d_2,\vp) \le N(\pi,\vp)$.
By Lemma \ref{el1} and \eqref{ab}
  $$m_G(  \{ \Re\chi_\pi >1-2\vp^2\}) \ge k^{-1}(\vp/2\pi)^d.$$
   Let $r=\|\chi_\pi\|_{\psi_2}$. Note $\|\Re\chi_\pi\|_{\psi_2}\le r$.
    We have for any $s>0$
    $$m\{ \Re\chi_\pi >s\}\le e \exp{-(s^2/r^2)},$$
    and hence with $s=d(1-2\vp^2)$
    we find
    $$k^{-1}(\vp/2\pi)^d\le e \exp{-(d^2(1-2\vp^2)^2/r^2)}.$$
    Choose say $\vp=1/2$. Then 
    a simple calculation leads
    to the announced lower bound for $r$.
     \end{proof}
   \begin{thm}\label{t10}
   If $G$ is finite or amenable (as a discrete group)
   then for any $d>1$ and any representation
   $\pi:\ G\to U(d)$  we have
   \begin{equation}\label{pe7-}\int_{U(d)} \sup\nolimits_{g\in G} | \tr ({u}  \pi(g)) | m_{U(d)}(d{u})\le c_1 \sqrt{d\log(d) }\end{equation}
   and
   \begin{equation}\label{pe7}  \|\chi_{ \pi}\| _{{{\psi_2}  }} \ge c_2 \sqrt{d/\log(d)} ,\end{equation}
    where $c_1$ and $c_2$ are positive constants independent of $d$.
         \end{thm}
    \begin{proof}  This follows
    from Proposition \ref{tb} and
      \eqref{e5}
    and  \eqref{e6} applied with $k=(d+1)!$.
     \end{proof}
     \begin{rem} The  proof of \eqref{pe7} is similar to but simpler than
      the one used  by Hutchinson for profinite groups in \cite{Hut},
      but since he used a weaker bound for $f(d)$ his estimate is weaker.
      
     \end{rem}
       \begin{rem} The estimates \eqref{pe7-} and  \eqref{pe7}
       are asymptotically optimal. This can be seen
       by considering the same case study $\cl G\subset U(d)$ as in
       Remark \ref{pr1}. Indeed,
       let  
       $Z_0=\sup\nl_{\sigma\in S(d)} |\sum\nl_{1}^d g_d(i,\sigma(i))|$.
       We have
      $$   Z_0\le \sup\nolimits_{v\in \cl G} | \tr ({u}   v ) | .$$
      Let  $V_\sigma=\sum\nl_{1}^d g_d(i,\sigma(i))$.
     Now if $\sigma, \sigma'\in S(d)$ differ
     on exactly $k$ places we have
     $$ \|V_\sigma-V_{\sigma'}\|_2^2=2k/d.$$
     Therefore if $r(\sigma,\sigma')$ denotes the number
     of $i$'s where $\sigma(i)\not=\sigma'(i)$  
     and if $u_\sigma$ denotes the unitary matrix
     associated to $\sigma$, we have
     $$\|V_\sigma-V_{\sigma'}\|_2 =(2r(\sigma,\sigma')/d)^{1/2}=\d_2(u_\sigma,u_{\sigma'}).$$
     Note that $r(\sigma,1) = d-j$ iff $\sigma$ has  exactly $j$ fixed points.\\
    We claim that the Sudakov minoration implies
    $\E Z_0\ge c_3 \sqrt{d \log d}$ for some $c_3>0$ independent of $d$.
        Indeed, applying \eqref{pe29}
        to the copy of $S(d)$ formed by the subgroup $G_S=\{u_\sigma\mid \sigma \in S(d)\}\subset U(d)$,  we find
        $$1/ m_{S(d)} (\{\sigma \in S(d)\mid  (2r(\sigma,1)/d)^{1/2} <\vp\} )\le N(G_S,\d_2,\vp).$$
        By \eqref{pe12}
        $$m_{S(d)} (\{\sigma \in S(d)\mid  (2r(\sigma,1)/d)^{1/2} <\vp\} )=
        \sum\nl_{j >d(1-\vp^2 /2)} X_j/d!\le \sum\nl_{j > d(1-\vp^2 /2)}  1/ j! .$$
        Note that for any $1\le k\le d$
        \begin{equation}\label{pe13}  \sum\nl_{j > k-1}  1/ j! \le e/ k!  \le e (e/k)^k.\end{equation}
        Taking e.g. $\vp=1$ and $k-1=[d/2]$, 
        the latter sum is $\le  e (2e/d)^{d/2}$ thus 
        $$ e^{-1}(d/2e)^{d/2} \le N(G_S,\d_2, 1).$$
        Note that
        $$\sup\nl_{u\in G_S} |\tr(g_d u) |=\sup\nl_\sigma|V_\sigma|=Z_0,$$
        and hence by \eqref{pe3} and  \eqref{sud}
$$   (b_1/\sqrt{2e} ) \sqrt{d \log d}\approx   b_1 \sqrt{ \log ( e^{-1}(d/2e)^{d/2}) } \le \E Z_0 . $$
This proves our claim and a fortiori
that $ \E\sup\nolimits_{v\in {\cl G}} | \tr (g_d   v) |\ge c_3 \sqrt{d \log d}$.
By the equivalence of $M_g$ and $M_u$ observed after \eqref{pe3'} this   proves 
that \eqref{pe7-} is optimal.

We now turn to \eqref{pe7}. For any  $u=\sum\nl_1^d \vp_i e_{i,\sigma(i)} \in\cl G$, let
$\chi(u) =\tr(u)=  \sum\nl_{ i\in {\rm Fix}(\sigma)}  \vp_i$.
Then for any $s>0$
$$m_{\cl G}(\{  |\chi| >s\})= 2 m_{\cl G}(\{  \chi >s\})\le  (2/d!) \sum\nl_{j>s}  X_j\le  2 \sum\nl_{j>s}   1/ j! .$$
We claim that there are positive constants (independent of $d$)
$c_4,c_5$ and $ c_6>e$ such that 
$$\forall s>c_6\quad m_{\cl G}(\{  |\chi| >s\})\le c_4 \exp{ -(c_5 s\log s)}.$$
Indeed, this is easy to derive from \eqref{pe13}.
Now since $m_{\cl G}(\{  |\chi| >s\})=0$ for all $s>d$
we may restrict consideration to $e<s<d$ for which
$s\log s = s^2 (\log s/s) \ge s^2 (\log d/d)  $, and we have
automatically
$$\forall s>c_6\quad m_{\cl G}(\{  |\chi| >s\})\le c_4 \exp{ -(c_5 s^2 (\log d/d))}.$$
Setting $F=\chi (d/ \log d)^{-1/2}$, we find
$$\forall s>c_6\quad m_{\cl G}(\{  |F| >s\})\le c_4 \exp{ -(c_5 s^2  )},$$
from which it is easy to deduce that,
for some $c_7$, we have $\|F\|_{\psi_2} \le c_7$, or equivalently
$$\|\chi\|_{\psi_2} \le c_7\sqrt{d/ \log d},$$ i.e. the case $G=\cl G$
shows that the growth when $d\to\infty$ of the constant in \eqref{pe7} is   optimal.
 \end{rem}   
     When the representations $\pi_n$ are defined on a single compact group
     $G$ (so that $G_n=G$ for all $n$), the next result was proved
     in \cite{Ce} for $G$ a   Lie group
     and in \cite{Hut} for $G$ a profinite group.
     \begin{cor}\label{ceb} If the groups $(G_n)$ appearing in Theorem \ref{t4} 
     are finite (or amenable as discrete groups) then the equivalent properties
     in Theorem \ref{t4}  can hold only if the dimensions $d_n=\dim(\pi_n)$ remain bounded.
     \end{cor}
     \begin{rem}\label{bf} The proof of the bound $d+1!$ in \cite{collins} uses the classification of finite simple groups. However,
     all that is needed for the last Corollary is a bound of
     the index that is $\exp{o(d^2)}$. Such a bound,  a much easier one,
     of the order $d^{c(d/log d)^2}=\exp{(cd^2/log d)}$ is known.
     It is due to   Blichfeldt, as indicated
      in 
      \cite[p. 103]{Bli} and \cite[p. 177]{Do}. 
Blichfeldt's bound  
      improved previous ones due to himself, then
      Bieberbach and Frobenius \cite{Fro}.       
        See  
        \cite[\S 36]{CuRe} for more on the subject.
    See \cite{BreJ} for a discussion of Jordan's ideas, and \cite{BreG} for more recent related results.
        \end{rem}
          Let $f_s(d)$ be the best possible $f(d)$ if one restricts to \emph{solvable} 
         finite subgroups  $G\subset U(d)$.
          In \cite[p. 218]{Do}  L. Dornhoff proves that   $f_s(d)=\exp {O(d)}$ and that this is optimal.
          Thus, for such groups $G$, we obtain a better bound:
          \begin{cor}
          There is a numerical constant $C>0$ such that for any $d$ and  any {solvable} 
         finite subgroup  $G\subset U(d)$ we have
         $$    \EE\sup\nl_{x\in G}|\tr(x g_d  )|\le  C \sqrt{d}.$$
          \end{cor}
          \begin{rem} The preceding bound is essentially optimal
          since if $G$ is the diagonal (finite Abelian) subgroup of $U(d)$
          with entries $=\pm1$
          we have clearly $    \EE\sup\nl_{x\in G}|\tr(x g_d  )|\approx  \sqrt{d}.$
              \end{rem}
       
       \section{Sidon sets}\label{sid}
Let $G$ be a compact group.  
\begin{dfn} A subset $\Lambda \subset \hat G$ is said to be a Sidon set if there is a constant $C$ such that
for any family $(a_\pi)$ with $a_\pi \in M_{d_\pi}$ and $\pi \mapsto a_\pi $
finitely supported, we have
\begin{equation}\label{se0}
\sum   \tr |a_\pi| \le C \sup\nl_{t\in G}| \sum   \tr(a_\pi \pi(t))|.\end{equation}
The smallest such $C$ is called the Sidon constant of $\Lambda$.\\
Let $\bb G=\prod\nl_{\pi\in \hat G} U(d_\pi) $.
We say that $\Lambda \subset \hat G$ is  randomly Sidon if
there is a constant $C$ such that
for any family $(a_\pi)$  as before we have
\begin{equation}\label{se00}
\sum   \tr |a_\pi| \le C \int_{\bb G} \sup\nl_{t\in G}| \sum   \tr(a_\pi u_\pi \pi(t))| m_{\bb G}(du).\end{equation}
The set $\Lambda$ is called local Sidon (resp. local randomly Sidon)
if \eqref{se0}  (resp. \eqref{se00}) only holds 
for all $(a_\pi)$ with at most a single non zero term.
(Note that these local variants are trivial in the commutative case.)
\end{dfn}
See \cite{FT} and \cite{HR} for early results on random Fourier series and
lacunary sets. See \cite{GH}
for a   more recent account on Sidon sets.

Obviously Sidon implies randomly Sidon. The converse was announced
by Rider in \cite{Ri} and proved there in the 
commutative case, but the first (and apparently only) published proof 
for the non-commutative case appeared only recently in \cite{Piri}. It follows
automatically that  local Sidon and  local randomly  Sidon are also equivalent properties.

Fix $1<p<\infty$.
The set $\Lambda \subset \hat G$ is called a $\Lambda(p)$-set
if there is a constant $C$ such that
for any family $(a_\pi)$  as before 
the function $F(t)= \sum  \tr(a_\pi \pi(t) )$
satisfies
\begin{equation}\label{se1}\| F \|_{p} \le C\| F \|_{1}.\end{equation}
When $p>1$ we can replace $\| F \|_{1}$ by $\| F \|_{2}$ in this definition.
See \cite{Bo2} for more  on $\Lambda(p)$-sets.

Using more recent terminology and recalling \eqref{ep66}, let us say   for short that 
$\Lambda \subset \hat G$ is subgaussian if there is a constant $C$ such that
  any $F$  as before 
 satisfies 
\begin{equation}\label{se2}\| F \|_{\psi_2} \le C\| F \|_{2}.\end{equation}
Using \eqref{se3}  this can be related to $\Lambda(p)$-sets.
The set $\Lambda$ is called  local $\Lambda(p)$  (resp. local  subgaussian)
if,  for some $C$, \eqref{se1} 
(resp. \eqref{se2}) holds for all $F$ of the form $F=\tr(a_\pi \pi(t) )$
with $\pi\in \Lambda$.\\
The adjective ``central" is added to  any one of the preceding
definitions to designate the property obtained
by restricting it to families $(a_\pi)$
 formed of scalar multiples of the identity (see \cite{Park}).
 
It was proved by the second author 
that subgaussian  implies Sidon. Since
the converse was already known (due to Rudin \cite{Ru} in the commutative case
and to   Fig\`a-Talamanca and  Rider \cite{FTR1} in the non-commutative one),
  Sidon and subgaussian are equivalent, and similarly for the local properties.
  See \cite{Pi2,Pi5, Piri} for more on this.
  
  Note that (i) in Theorem \ref{t4}
  means equivalently that the set formed of the coordinates 
  on $\prod\nl_{n\ge 1} G_n$ is a local Sidon set, while (ii)
  means that it is  a central local subgaussian set,
  and   (iii)
  means that it is  a central  local randomly Sidon set. Actually using 
  an averaging argument based on the irreducibilty of the $\pi$'s, it is rather easy to show that a
  central  local randomly Sidon set is local randomly Sidon. 
 
   In the non-commutative case, these notions took a serious stepback
   when it was discovered that for   most classical
  compact groups $G$  there are \emph{no} infinite 
  subsets $\Lambda \subset \hat G$  satisfying them except in the case
when the dimensions $(d_\pi)_{\pi \in \Lambda}$
are bounded. More precisely, Cecchini \cite{Ce}
proved that there are no infinite $\Lambda(4)$-sets
in $\hat G$ with unbounded  dimensions if $G$ is a compact Lie group.
Giulini and Travaglini \cite{GT} improving results
due to Price and Rider
 proved that for any compact connected semisimple Lie
group $G$ there are no infinite local $\Lambda_p$ sets for $p>1$.
Related results appear in \cite{Park,Ri3,Ri4,Doo,Ha}.
Cartwright and McMullen \cite{CaMc} 
characterized the compact connected groups
that admit an infinite local Sidon set,
and proved that they contain
an infinite Sidon set. 
Hutchinson \cite{Hut} proved that there are no
infinite central local subgaussian sets with unbounded  dimensions for $G$ profinite. 

Following the recent paper \cite{BoLe}
the second author investigated what remains of Theorem \ref{t4}
when one replaces $t\mapsto \pi_n(t)$
by a matrix valued function $t\mapsto \varphi_n(t)$ 
on an arbitrary probability space $(T,m)$
satisfying the same moment conditions
as $t\mapsto \pi_n(t)$, namely the following:
\begin{equation}\label{31}
 \exists C' \ \forall n\quad \|\varphi_n\|_{L_\infty(M_{d_n})} \le C'
 \end{equation}   
\begin{equation}\label{31'}\forall i,j,k,l\quad \int \varphi_n(i,j) \ovl{\varphi_{n}(k,\ell)}dm=d_n^{-1}  \delta_{i,k}\delta_{j,\ell}.\end{equation}
In other words, $\{d_n^{1/2}\varphi_n(i,j)\mid   1\le i,j\le d_n\}$
 is an orthonormal system for each $n$.
The analogue of the local subgaussian condition
is then:
\begin{equation}\label{32}
 \exists C''\  \forall n\ge 1\ \forall a\in M_{d_n}\quad
 \|  \tr (a \varphi _n)\|_{\psi_2}\le C''\|    \tr (a \varphi_n)\|_{2}=C'' (    {d_n}^{-1}\tr |a|^2)^{1/2}. 
\end{equation}
Under these conditions, there is a constant $C$ (depending only
on $C',C''$) such that 
\begin{equation}\label{33}\forall n\ge 1\ \forall a\in M_{d_n}\quad
 \tr |a| \le C \sup\nl_{t_1,t_2\in T}|     \tr(a \varphi_n(t_1)\varphi_n(t_2))|.\end{equation}
This generalizes the implication
local subgaussian $\Rightarrow$ local Sidon mentioned above for representations.
 This is proved in \cite[Remark 3.14]{Pi3}. 
 Obviously, 
 in this general setting there is no obstruction
preventing $\varphi$ from having a \emph{finite range}. Nevertheless,
if 
the range of $\varphi$
is in some sense close to a group
it is natural
to expect that an analogue of
 Corollary \ref{ceb} holds.
 For instance, fix $\vp>0$ and   $\chi\ge 1$. Assume that
  there is a  subgroup $G_n\subset GL(d_n)$, amenable as a discrete group,
 such that
\begin{equation}\label{34}\forall n\ge 1\ 
\sup\nl_{u\in G_n} \|u\|\le \chi  \text{  and  } 
\forall t\in T \quad \exists u \in G_n\text{ such that } 
\| \varphi_n(t)- u\|\le \vp .\end{equation}
Then, here is
 one possible generalization of
 Corollary \ref{ceb}:
 \begin{cor} Assume \eqref{31} \eqref{33} and \eqref{34}.
If $\vp< (C(C' +\chi))^{-1}$, then  $\sup\nl_n d_n<\infty$.
 \end{cor}
  \begin{proof}
   Let $\d=C^{-1}-( C' +\chi) \vp$. By our assumption   $\d>0$. We first claim that
 $$\forall n\ge 1\ \forall a\in M_{d_n}\quad \d\tr|a|\le \sup\nl_{u\in G_n} 
 |\tr(u a)|.$$
  Indeed, by \eqref{31} and \eqref{34} for any $t_1,t_2\in T$ 
  $$|     \tr(a \varphi_n(t_1)\varphi_n(t_2))|
  \le |     \tr(a (\varphi_n(t_1)-u_1)\varphi_n(t_2))|
  +|\tr(a u_1(\varphi_n(t_2)-u_2))| +|\tr(a u_1 u_2)|
   $$
  and hence
$$ |     \tr(a \varphi_n(t_1)\varphi_n(t_2))|\le (C' +\chi) \vp \tr|a|+ \sup\nl_{u_1,u_2\in G_n} 
 |\tr(u_1u_2 a)|=  (C' +\chi) \vp \tr|a|+ \sup\nl_{u \in G_n} 
 |\tr(u  a)|.$$
 From this and \eqref{33} the claim is immediate. By a well known averaging argument,
 since $G_n$ is amenable, there is $v\in GL(d_n)$
 with $\|v\| \|v^{-1}\| \le \chi^2$ such
 that $G_n'=v G_n v^{-1}\subset U(d_n)$. Then our claim implies
 $$ \d\tr| v^{-1} a v|\le \sup\nl_{u\in G'_n} 
 |\tr(u a)|,$$ and hence
   $ \d \chi^{-2} \tr|a| \le \d \tr| v^{-1} a v|\le \sup\nl_{u\in G'_n} 
 |\tr(u a)|.$ Thus we conclude by Corollary \ref{ceb} and Remark \ref{pr50}.
  \end{proof}

\begin{rem}  In a paper in preparation we plan to
give a characterization of the 
sequences of irreducible unitary representations $\pi_n:\ G_n \to U(d_n)$
for which Theorem \ref{t4} holds. 
We will give a structural description of
the irreducible compact subgroups 
of $U(d)$ for which the inclusion
$\pi:\ G\to U(d)$ has   a bounded constant $C_2(\pi)$  (with the notation
in Lemma \ref{fl1}).

\end{rem}


\begin{thebibliography}{99}
  
  
  
  
 

       \bibitem{Bli} H.F. Blichfeldt, \emph{Finite collineation groups},
      University of Chicago Press, Chicago, 1917.

     \bibitem{borel-serre} A. Borel and J.P. Serre,  \emph{Th\'eor\`emes de finitude en cohomologie galoisienne}, Comm. Math. Helv. vol. 39, 1964--1965.

 \bibitem{Bo2} J. Bourgain, $\Lambda_p$-sets in analysis: results, problems and related aspects. Handbook of the geometry of Banach spaces, Vol. I, 195--232, North-Holland, Amsterdam, 2001. 
  
  \bibitem{BoLe} J. Bourgain and M. Lewko, Sidonicity and variants of Kaczmarz's problem, preprint, arxiv, April 2015.
 
   
  
 \bibitem{BreJ} E. Breuillard,  An exposition of Camille Jordan's original proof of his theorem on finite subgroups of invertible matrices, notes available at \\http://www.math.u-psud.fr/breuilla/cour.html.
 
 \bibitem{BreG} E. Breuillard and B. Green,
 Approximate groups III: the unitary case,
  
   
   \bibitem{Ca} B. Carl,
  Entropy numbers of diagonal operators with an application to eigenvalue problems. J. Approx. Theory 32 (1981),  135--150. 
    
   \bibitem{CaMc} D. Cartwright and J. McMullen,  
A structural criterion for the existence of infinite Sidon sets.
Pacific J. Math. 96 (1981),   301--317. 
    
   \bibitem{Ce} C. Cecchini, Lacunary Fourier series on compact Lie groups. J. Funct. Anal. 11 (1972) 191--203.
 
 \bibitem{collins} M. Collins,  {On Jordan's theorem for complex linear groups}, J. Group Theory 10 (2007), 411--423.

\bibitem{collins2} M. Collins,  {Bounds for finite primitive complex linear groups}, Journal of Algebra 319 (2008) 759--776.


    \bibitem{CuRe} C. Curtis and  I. Reiner, \emph{Representation theory of finite groups and associative algebras}, Interscience, New-York, 1962.

 \bibitem{Doo}  A. Dooley,   Norms of characters and lacunarity for compact Lie groups. J. Funct. Anal. 32 (1979), 254--267.

\bibitem{Do}  
 L. Dornhoff, \emph{Group representation theory, part A, Ordinary
 representation theory}. Marcel Dekker, New-York, 1971.  
 
   
     \bibitem{FT} A. Fig\`a-Talamanca, Random Fourier series on compact groups. Theory of Group Representations and Fourier Analysis (C.I.M.E., II Ciclo, Montecatini Terme, 1970) pp. 1--63 Edizioni Cremonese, Rome, 1971.
     
      \bibitem{FTR1} A. Fig\`a-Talamanca and D. Rider,  
  A theorem of Littlewood and lacunary series for compact groups. Pacific J. Math. 16 (1966) 505--514. 
   
 \bibitem{Fro}  G. Frobenius, \"Uber den L. Bieberbach gefundenen Beweis eines Satzes von C. Jordan, Sitzber. Preuss. Akad. Wiss.
(1911) 241--248.
   
     \bibitem{GT}   S. Giulini
and
G. Travaglini,  
$L_p$-estimates for matrix coefficients of irreducible representations of compact
groups.
Proc. Amer. Math. Soc.
80 (1980), 448--450.
   
    \bibitem{GH} 
  C. Graham and K. Hare,  
\emph{Interpolation and Sidon sets for compact groups}.
Springer, New York, 2013. xviii+249 pp. 
    
         \bibitem{Ha} K.  Hare,   Central Sidonicity for compact Lie groups. Ann. Inst. Fourier (Grenoble) 45 (1995),   547--564.
 
   \bibitem{HR}  E. Hewitt and K. Ross, \emph{Abstract harmonic analysis, Volume II,
Structure and Analysis for Compact Groups,
Analysis on Locally Compact Abelian Groups}, Springer,  Heidelberg, 1970.
          
           
    \bibitem{Hut} M. Hutchinson, Local $\Lambda$ sets for profinite groups, Pacific J.  Math.  80 (1980) 81--88.        
  
  \bibitem{Le} M. Ledoux, \emph{ The concentration of measure phenomenon}, Mathematical Surveys and Monographs, 89. American Mathematical Society, Providence, RI, 2001.  
  
 \bibitem{LeTa} M. Ledoux and M. Talagrand,  
\emph{Probability in Banach Spaces. Isoperimetry and Processes},
Springer-Verlag, Berlin, 1991. 
 
 
 
\bibitem{MaPi}  M.B. Marcus and G. Pisier, 
\emph{Random Fourier series with Applications to
Harmonic Analysis.}   Annals
of Math. Studies n$^\circ$101, Princeton Univ. Press,
1981.  
  
  

\bibitem{Park}   W. A. Parker,   Central Sidon and central $\Lambda_p$ sets. J. Austral. Math. Soc. 14,
62--74 (1972).
   
 \bibitem{Pi2}  G. Pisier, De nouvelles caract{\' e}risations des ensembles
de Sidon.  Advances in Maths. Supplementary studies, vol
7B (1981) 685--726.
 
 \bibitem{Pi5}  G. Pisier, Condition d'entropie et caract{\' e}risations
arithm{\' e}tiques des ensembles de Sidon.  Modern Topics
in Harmonic Analysis - Torino/Milano - June/July 1982,
Inst. di Alta Math - Rome (1983) vol. II., 911--944.
 
  \bibitem{Pip} G. Pisier,  Probabilistic methods in the geometry of Banach spaces,  Probability and analysis (Varenna, 1985),  167--241, \emph{Lecture Notes in Math.}    {   1206}, Springer-Verlag, Berlin, 1986. 

 \bibitem{Piv}  G. Pisier, \emph{The volume of Convex Bodies and Banach
Space Geometry}.  Cambridge University Press, 1989.

   
    \bibitem{Pi3}  G. Pisier, On  
 uniformly  bounded orthonormal Sidon systems, preprint, arxiv 2016.
 To appear in Math. Res. Letters.
  
  \bibitem{Piri} G. Pisier, Spectral gap properties of the unitary groups:
around Rider's results on non-commutative Sidon sets, arxiv 2016,
to appear.
  
  \bibitem{Ri}  D. Rider, Randomly continuous functions and Sidon sets.
 Duke Math. J. 42 (1975) 752--764.
 
 
  \bibitem{Ri3} D. Rider, Norms of characters and central $\Lambda_p$ sets for $U(n)$. Conference on Harmonic Analysis (Univ. Maryland, College Park, Md., 1971), pp. 287--294. Lecture Notes in Math., Vol. 266, Springer, Berlin, 1972.
 
  \bibitem{Ri4} D. Rider, Central lacunary sets. Monatsh. Math. 76 (1972), 328--338.
 
  \bibitem{Ru} W. Rudin, Trigonometric series with gaps.
   J.  Math. and Mech. 9 (1960) 203--227.
  
    \bibitem{Sta}  R. Stanley, \emph{Enumerative combinatorics, vol. 1} Cambridge Univ. Press,
    Cambridge, 1997.
    
     \bibitem{Ta2} M. Talagrand,  \emph{Upper and Lower Bounds for
Stochastic Processes}, Springer, Berlin, 2014.

 \bibitem{tits} J. Tits,  {Free Subgroups in Linear Groups}, Journal of  Agebra 20, (1972) 250--270.

   
  \bibitem{Tu} A. Turing, Finite approximations to Lie groups,
 Annals  of Math.
 39 (1938), 105--111.
 
   
 \bibitem{wehrfritz} B. A. F. Wehrfritz, \emph{Infinite linear groups}, Springer-Verlag, New York-Heidelberg, (1973), Ergebnisse
der Matematik und ihrer Grenzgebiete, Band 76.
   
\bibitem{weisfeiler} B. Weisfeiler, Post-classification version of Jordan's theorem on finite
linear groups, Proc. Natl. Acad. Sci. USA
  81 (1984), 5278--5279.
 
  

 \end{thebibliography}
 \end{document}